\documentclass[final]{article}
\usepackage{amsmath,mathrsfs,amssymb,amsthm,amscd}
\usepackage{color,epsfig,latexsym,graphicx,eucal,layout,fancyhdr}
\usepackage{enumerate}
\usepackage[]{fontenc}
\usepackage[all]{xy}
\usepackage{hyperref}
\usepackage{graphics}
\usepackage{a4wide,verbatim}
\usepackage{psfrag}

\newcounter{EQNR}
\renewcommand{\contentsname}{Contents\\
{\footnotesize\normalfont(A table of contents should normally not be included)}}

\newtheorem{theorem}{Theorem}

\newtheorem{corollary}[theorem]{Corollary}

\newtheorem{example}[theorem]{Example}

\newtheorem{lemma}[theorem]{Lemma}

\newtheorem{proposition}[theorem]{Proposition}
\newtheorem{remark}[theorem]{Remark}

\DeclareMathOperator{\PSL}{PSL}
\DeclareMathOperator{\hyp}{hyp}
\let\Im\relax
\DeclareMathOperator{\Im}{Im}
\let\Re\relax
\DeclareMathOperator{\Re}{Re}
\newcommand{\C}{\mathbb{C}}
\newcommand{\h}{\mathbb{H}}

\newcommand{\R}{\mathbb{R}}
\newcommand{\Z}{\mathbb{Z}}
\newcommand{\E}{\mathcal{E}}
\newcommand{\F}{\mathcal{F}}
\newcommand{\eps}{\varepsilon}
\newcommand{\abs}[1]{\left\vert#1\right\vert}
\DeclareMathOperator{\SL}{SL}
\DeclareMathOperator{\vol}{Vol}

\begin{document}

\title{Applications of Kronecker's limit formula for elliptic Eisenstein series}
\author{Jay Jorgenson \and Anna-Maria von Pippich \and Lejla Smajlovi\'{c}
\footnote{\noindent The first named author acknowledges grant support from the NSF and PSC-CUNY.
We thank Professor Floyd Williams for making available to us the unpublished dissertation \cite{Vassileva96}
which was written by his student I. N. Vassileva.  The results in the manuscript were of great interest to us, and
we hope the document will become available to the mathematical community.}}
\maketitle

\begin{abstract}\noindent
We develop two applications of the Kronecker's limit formula associated
to elliptic Eisenstein series:  A factorization theorem for holomorphic modular forms, and a
proof of Weil's reciprocity law. Several examples of the general factorization results are computed, specifically for
certain moonshine groups, congruence subgroups, and, more generally, non-compact subgroups with one cusp.
In particular, we explicitly compute the Kronecker limit function associated to certain elliptic points for a
few small level moonshine groups.
\end{abstract}

\vskip .15in
\section{Introduction and statement of results}

\subsection{Non-holomorphic Eisenstein series.}
Let $\Gamma$ be a Fuchsian group of the first kind which acts on the hyperbolic
space $\mathbb H$ by factional linear transformations, and let $M = \Gamma \backslash \mathbb H$
be the finite volume quotient.  One can view $M$ as a finite volume hyperbolic Riemann
surface, possibly with cusps and elliptic fixed points.  In a slight abuse of notation, we will
use $M$ to denote both the Riemann surface as well as a (Ford) fundamental domain of $\Gamma$
acting on $\mathbb H$.

The abelian subgroups of $\Gamma$ are classified as three distinct types:  Parabolic, hyperbolic and
elliptic.  Accordingly, there are three types of scalar-valued non-holomorphic Eisenstein series, whose
definitions we now recall.

Parabolic subgroups are characterized by having a unique fixed point $P$ on the extended upper-half plane
$\widehat{\mathbb H}$.  The fixed point $P$ is known as a cusp of $M$, and the associated parabolic subgroup
is denoted by $\Gamma_{P}$.
The parabolic Eisenstein series ${\cal E}^{\mathrm{par}}_{P}(z,s)$ associated to $P$ is a defined for
$z\in M$ and $s \in \mathbb{C}$ with $\textrm{Re}(s) > 1$, by the series
\begin{equation*}
{\cal E}^{\mathrm{par}}_{P}(z,s) =
\sum\limits_{\eta \in \Gamma_{P}\backslash \Gamma}\textrm{Im}(\sigma_{P}^{-1}\eta z)^{s},
\end{equation*}
where $\sigma_{P}$
is the scaling matrix for the cusp $P$, i.e. the element of $\mathrm{PSL}_2(\mathbb{R})$ such that, when extending the action of $\sigma_{P}$ to $\widehat{\mathbb H}$,
we have that $\sigma_{P}\infty = P$.

Hyperbolic subgroups have two fixed points on the extended upper-half plane $\widehat{\mathbb H}$.  Let us denote
a hyperbolic subgroup by $\Gamma_{\gamma}$ for $\gamma \in \Gamma$, and let $\mathcal{L}_{\gamma}$ signify the
geodesic path in $\mathbb H$ connecting the two fixed points of hyperbolic element $\gamma$.  Following Kudla and Millson from \cite{KM79},
one defines a scalar-valued hyperbolic Eisenstein series for $z\in M$ and $s \in \mathbb{C}$ with
$\textrm{Re}(s) > 1$ by the series
\begin{equation}\label{hyp_eisen}
{\cal E}^{\mathrm{hyp}}_{\gamma}(z,s) =\sum\limits_{\eta \in \Gamma_{\gamma}\backslash \Gamma}
\cosh(d_{\mathrm{hyp}}(\eta z, \mathcal{L}_{\gamma}))^{-s},
\end{equation}
where $d_{\mathrm{hyp}}(\eta z, \mathcal{L}_{\gamma})$ is the hyperbolic distance from the
point $\eta z$ to  $\mathcal{L}_{\gamma}$.

Elliptic subgroups have finite order and have a unique fixed point within $\mathbb H$.  In fact,
for any point $w \in M$, there is an elliptic subgroup $\Gamma_{w}$ which fixes $w$, where in
all but a finite number of cases $\Gamma_{w}$ is the identity element.  Elliptic Eisenstein series
were defined in an unpublished manuscript from 2004 by Jorgenson and Kramer and were studied in depth
in the 2010 dissertation \cite{vP10} by von Pippich.  Specifically, for $z\in M$, $z\not=w$, and
$s \in \mathbb{C}$ with $\textrm{Re}(s) > 1$, the elliptic Eisenstein series is defined by
\begin{equation}\label{ell_eisen}
{\cal E}^{\textrm{ell}}_{w}(z,s) =\sum\limits_{\eta \in \Gamma_{w}\backslash \Gamma}
\sinh(d_{\mathrm{hyp}}(\eta z, w))^{-s}
\end{equation}
where $d_{\mathrm{hyp}}(\eta z, w)$ denotes the hyperbolic distance from $\eta z$ to $w$.

\subsection{Known properties and relations}

There are some fundamental differences between the three types of Eisenstein series defined above.
Hyperbolic Eisenstein series are in $L^{2}(M)$, whereas parabolic and elliptic series are not.
Elliptic Eisenstein series are defined as a sum over a finite index subset of $\Gamma$, and indeed
the series (\ref{ell_eisen}) can be extended to all $\Gamma$ which would introduce a multiplicative
factor equal to the order of $\Gamma_{w}$.  However, hyperbolic and parabolic series are necessarily
formed by sums over infinite index subsets of $\Gamma$.  Parabolic Eisenstein series are eigenfunctions
of the hyperbolic Laplacian; however, elliptic and hyperbolic Eisenstein series satisfy a differential-difference
equation which involves the value of the series at $s+2$.

Despite their differences, there are several intriguing ways in which the Eisenstein series interact.  Since the hyperbolic
Eisenstein series are in $L^{2}(M)$, the expression (\ref{hyp_eisen}) admits a spectral expansion which involves the
parabolic Eisenstein series; see \cite{JKvP10} and \cite{KM79}.  If one considers a degenerating
sequence of Riemann surfaces obtained by pinching a geodesic, then the associated hyperbolic Eisenstein
series converges to parabolic Eisenstein series on the limit surface; see \cite{Fa07} and \cite{GJM08}.
If one studies a family of elliptically degenerating surfaces obtained by re-uniformizing at a point
with increasing order, then the corresponding elliptic Eisenstein series converge to parabolic Eisenstein series on
the limit surface; see \cite{GvP09}.

Finally, there are some basic similarities amongst the series.  Each series admits a meromorphic continuation
to all $s\in\mathbb{C}$. The poles of the meromorphic continuations have been identified and are closely related,
in all cases involving data associated to the continuous and non-cuspidal discrete spectrum of the
hyperbolic Laplacian and, for hyperbolic and elliptic series, involving data associated to the cuspidal
spectrum as well.  Finally, and most importantly for this article, the hyperbolic and elliptic Eisenstein series are holomorphic at
$s=0$, and for all known instances, the parabolic Eisenstein series also is holomorphic at $s=0$. In all these
cases, the value of each Eisenstein series at $s=0$ is a constant as a function of $z$.
The coefficient of $s$ in the Taylor series expansion about $s=0$
shall be called the Kronecker limit function.

\subsection{Kronecker limit functions}

The classical Kronecker's limit formula is the following statement, which we
quote from \cite{Siegel80}. If we consider the case when $\Gamma = \textrm{PSL}_2(\mathbb{Z})$,
then
\begin{align*}
\mathcal{E}^{\mathrm{par}}_{\infty}(z,s)=
\frac{3}{\pi(s-1)}
-\frac{1}{2\pi}\log\bigl(|\Delta(z)|\Im(z)^{6}\bigr)+C+O(s-1) \,\,\,\,\,\textrm{as $s \rightarrow 1$,}
\end{align*}
where $C=6(1-12\,\zeta'(-1)-\log(4\pi))/\pi$, and with Dedekind's delta function $\Delta(z)$ given by
$$
\Delta(z) = \left[q_{z}^{1/24}\prod\limits_{n=1}^{\infty}\left(1 - q_{z}^{n}\right)\right]^{24} = \eta(z)^{24}
\,\,\,\,\,\textrm{with $q_{z} = e^{2\pi i z}$.}
$$
By employing the well-known functional equation for $\mathcal{E}^{\mathrm{par}}_{\infty}(z,s)$,
Kronecker's limit formula can be reformulated as
\begin{equation*}
\mathcal{E}^{\mathrm{par}}_{\infty}(z,s)=
1+ \log\bigl(|\Delta(z)|^{1/6}\Im(z)\bigr)s+O(s^2) \,\,\,\,\,\textrm{as $s \rightarrow 0$.}
\end{equation*}

For general Fuchsian groups of the first kind, Goldstein \cite{Go73} studied analogue of Kronecker's limit formula
associated to parabolic Eisenstein series.  We will use the results from \cite{Go73} throughout this article.

The hyperbolic Eisenstein series in \cite{KM79} are form-valued, and
the series are defined by an infinite sum which converges for $\textrm{Re}(s) > 0$.  The main result in \cite{KM79}
is that the the form-valued hyperbolic Eisenstein series is holomorphic at $s=0$, and the value is equal to the harmonic form that is
the Poincar\'e dual to the one-cycle in the homology group $H^{1}(M,\mathbb R)$ corresponding to the hyperbolic geodesic
$\gamma$ fixed by $\Gamma_{\gamma}$.

The analogue of Kronecker's limit formula for elliptic Eisenstein series was first proved in \cite{vP10} and \cite{vP15}. Specifically,
it is shown that at $s=0$, the series (\ref{ell_eisen}) admits the Laurent expansion
\begin{align}\label{Kronecker_elliptic}
\mathrm{ord}(w)\,\mathcal{E}^{\mathrm{ell}}_{w}(z,s)&-
\frac{2^{s}\sqrt{\pi}\,\Gamma(s-\frac{1}{2})}{\Gamma(s)}
\sum\limits_{k=1}^{p_{\Gamma}}\mathcal{E}^{\mathrm{par}}_{p_{k}}(w,1-s)
\,\mathcal{E}^{\mathrm{par}}_{p_{k}}(z,s)= \notag \\
&-c
-\log\bigl(|H_{\Gamma}(z,w)|^{\mathrm{ord}(w)}(\Im(z))^{c}\bigr)\cdot s+O(s^2)\,\,\,\,\,\textrm{as $s \rightarrow 0$,}
\end{align}
where $p_k$, $k=1,\ldots, p_{\Gamma}$, are cusps of $M$, $c=2\pi/\vol_{\hyp}(M)$, and $H_{\Gamma}(z,w)$ is a holomorphic
automorphic function with respect to $\Gamma$ and which vanishes only when
$z=\eta w$ for some $\eta\in\Gamma$.  Two explicit computations are given in \cite{vP10} and \cite{vP15}
for $\Gamma = \PSL_2(\Z)$ when considering the elliptic Eisenstein series $E^{\mathrm{ell}}_{w}(z,s)$
associated to the points $w=i$ and $w=\rho = (1+i\sqrt{3})/2$.  In these cases, the elliptic Kronecker limit function
$H_{\Gamma}(z,w)$ at points $w=i$ and $w=\rho$ is such that

\begin{equation}\label{elliptic_at_i}
\abs{H_{\Gamma}(z,i)}= \exp(-B_i)\abs{E_6(z)}, \text{   where   } B_i=-3(24\zeta'(-1)-\log(2\pi)+4\log\Gamma(1/4))
\end{equation}
and
\begin{equation}\label{elliptic_at_rho}
\abs{H_{\Gamma}(z,\rho)}=\exp(-B_{\rho})\abs{E_4(z)}, \text{   where   } B_{\rho}=-2(24\zeta'(-1)-2\log(2\pi/\sqrt{3})+6\log\Gamma(1/3)).
\end{equation}

The Kronecker limit formula for elliptic Eisenstein series became the asymptotic formulas
\begin{equation}\label{elliptic Eis_at_i}
\E^{\mathrm{ell}}_{i}(z,s)= -\log(\vert E_{6}(z)\vert \vert\Delta(z)\vert^{-1/2})\cdot s + O(s^{2})
\,\,\,\,\,\text{\rm as $s\rightarrow 0$,}
\end{equation}
and
\begin{equation}\label{elliptic Eis_at_rho}
\E^{\mathrm{ell}}_{\rho}(z,s)= -\log(\vert E_{4}(z)\vert \vert\Delta(z)\vert^{-1/3})\cdot s + O(s^{2})
\,\,\,\,\,\text{\rm as $s\rightarrow 0$,}
\end{equation}
where $E_{4}$ and $E_{6}$ are classical holomorphic Eisenstein series on $\PSL_2(\Z)$ of weight four and six, respectively.

Before continuing, let us state what we believe to be an interesting side comment.  The Kronecker limit function associated to elliptic
Eisenstein series is naturally defined as the coefficient of $s$ in the Laurent expansion of the elliptic Eisenstein series near $s=0$.
As we show below, one can realize the Kronecker limit function for parabolic Eisenstein series for groups with one cusp as the coefficient
of $s$ in the Laurent expansion of the parabolic Eisenstein series at $s=0$.  One has yet to study the Laurent expansion near $s=0$, in particular
the coefficient of $s$, for the scalar-valued hyperbolic Eisenstein series; for that matter, we have not fully understood
the analoguous question for the vector of parabolic Eisenstein series for general groups.  We expect that one can develop systematic theory by
focusing on coefficients of $s$ in all cases.

\subsection{Important comment and assumption}\label{subsection_assumption}

At this time, we do not have a complete understanding of the behavior of the parabolic Eisenstein series
${\cal E}^{\mathrm{par}}_{P}(z,s)$ near $s=0$.  If the group has one cusp, the functional equation of the
Eisenstein series shows that ${\cal E}^{\mathrm{par}}_{P}(z,0)=1$.  In notation to be set below, its
scattering determinant is zero at $s=0$.  However, this is not true when there is more than one cusp.  For example,
on page 536 of \cite{He83}, the author computes the scattering matrix for $\Gamma_{0}(N)$ for square-free $N$, from
which it is clear that $\Phi(s)$ is holomorphic but not zero at $s=0$. Specifically, it remains to determine if the parabolic
Eisenstein series is holomorphic at $s=0$, which is a question we were unable to answer in complete generality.

\vskip .10in
\it Throughout this article, we assume that ${\cal E}^{\mathrm{par}}_{P}(z,s)$ is holomorphic at $s=0$.
\rm

\vskip .10in
\noindent
The assumption is true in all the instances where specific examples are developed.

\subsection{Main results}
The purpose of the present paper is to further study the Kronecker limit function associated to
elliptic Eisenstein series.  We develop two applications.  To begin, we examine the relation (\ref{Kronecker_elliptic})
and study the contribution near $s=0$ of the term involving the parabolic Eisenstein series.  As with the parabolic
Eisenstein series, the resulting expression is particularly simple in the case when the group $\Gamma$ has one cusp.
However, in all cases, we obtain an asymptotic formula for $\mathcal{E}^{\mathrm{ell}}_{w}(z,s)$ near $s=0$ which allows
us to prove asymptotic bounds for the elliptic Kronecker limit function in any parabolic cusp associated to $\Gamma$.
As a consequence, we are able to prove the main result of this article, namely a factorization theorem
which expresses holomorphic forms on $M$ of arbitrary weight as products of the elliptic Kronecker limit functions.

The product formulas are developed in detail in the case of so-called moonshine groups, which are discrete
groups obtained by adding the Fricke involutions to the congruence subgroups $\Gamma_{0}(N)$.  As an application
of the factorization theorem, we establish further examples of relations similar to (\ref{elliptic_at_i}),
(\ref{elliptic_at_rho}), \eqref{elliptic Eis_at_i} and \eqref{elliptic Eis_at_rho}.
For example, the moonshine group $\Gamma = \overline{\Gamma_0(2)^+} = \Gamma_0(2)^+/\{\pm \textrm{Id}\}$ has $e_{2}=1/2 + i/2$ as a fixed point of
order four.
In section 6.2, we prove that the elliptic Kronecker limit function $H_2(z,e_2)$ associated to the point $e_2$ is such that
$$
\abs{H_2(z,e_2)} = \exp(-B_{2,e_2})\abs{E_{4}^{(2)}(z)}^{1/2},
$$
where $E_{4}^{(2)}(z)$ is the weight four holomorphic Eisenstein series associated to $\Gamma_{0}(2)^{+}$ and
$$
B_{2,e_2}=- \left( 24\zeta'(-1) + \log(8\pi^2)
- \frac{11}{6} \log 2 +\frac{1}{12} \log\left( \left| \Delta(1/2 + i/2) \cdot \Delta(1+i) \right| \right)\right).
$$
In this case, the Kronecker limit formula for the elliptic Eisenstein series $\E^{\mathrm{ell}}_{e_{2}}(z,s)$ reads as
\begin{equation*}
\E^{\mathrm{ell}}_{e_{2}}(z,s)= -\log\left(\vert E_{4}^{(2)}(z)\vert^{1/2} \vert \Delta(z)\Delta(2z) \vert ^{-1/12}\right)\cdot s + O(s^{2})
\,\,\,\,\,\text{\rm as $s\rightarrow 0$,}
\end{equation*}
or, equivalently, as
\begin{equation} \label{ell Eis at e_2}
\E^{\mathrm{ell}}_{e_{2}}(z,s)= -\log\left(\frac{1}{\sqrt{5}}\vert E_{4}(z) + 4E_4(2z)\vert^{1/2} \vert \Delta(z)\Delta(2z) \vert ^{-1/12}\right)\cdot s + O(s^{2})
\,\,\,\,\,\text{\rm as $s\rightarrow 0$.}
\end{equation}

The factorization theorem allows one to formulate numerous of examples of this type, of which we develop a few
for certain moonshine and congruence subgroups.

Second, we use the elliptic Kronecker limit formula to give a new proof of Weil's reciprocity formula.  A number of
authors have obtained generalizations of Weil's reciprocity law; see, for example, the elegant presentation in \cite{Kh08} which
discusses various reciprocity laws over $\mathbb C$ as well as Deligne's article \cite{De91} where the author re-interprets Tate's
local symbol and obtains a number of generalizations and applications.  It would be interesting to study the possible connection between
the functional analytic method of the present and companion article \cite{JvPS14} with the algebraic ideas in \cite{De91} and
results surveyed in \cite{Kh08}.

An outline of this article is as follows.  In section 2 we establish notation and cite various results from the literature.  In
section 3, we reformulate Kronecker's limit formula for parabolic Eisenstein series as an asymptotic statement near $s=0$.  From the
results in section 3, we then prove, in section 4, the asymptotic behavior in the cusps of the elliptic Kronecker limit function.
Specific examples are given for moonshine groups $\overline{\Gamma_{0}(N)^{+}}$ with square-free level $N$ and congruence subgroups
$\overline{\Gamma_{0}(p)}$
with prime level $p$.  In section 5 we prove the factorization theorem which states, in somewhat vague terms, that any holomorphic
form on $M$ can be written as a product of elliptic Kronecker limit functions, up to a multiplicative constant.
In addition, from the asymptotic formula from section 4, one is able to obtain specific information associated to the multiplicative
constant in the aforementioned description of the factorization theorem.  In section 6 we give examples of the factorization theorem
for holomorphic Eisenstein series for the modular group, for moonshine groups of levels $2$ and $5$, for general moonshine groups,
and for congruence subgroups $\overline{\Gamma_{0}(p)}$ of prime level.  Finally, in section 7, we present our proof of Weil's reciprocity
using the elliptic Kronecker limit functions and state a few concluding remarks.

\section{Background material}

\subsection{Basic notation} \label{notation}
Let $\Gamma\subseteq\mathrm{PSL}_{2}(\mathbb{R})$ denote a Fuchsian
group of the first kind acting by fractional
linear transformations on the hyperbolic upper half-plane $\mathbb{H}:=\{z=x+iy\in\mathbb{C}\,
|\,x,y\in\mathbb{R};\,y>0\}$. We let $M:=\Gamma\backslash\mathbb{H}$, which is a finite
volume hyperbolic Riemann surface, and denote by $p:\mathbb{H}\longrightarrow M$
the natural projection. We assume that $M$ has $e_{\Gamma}$
elliptic fixed points and $p_{\Gamma}$ cusps. We identify $M$
locally with its universal cover $\mathbb{H}$.

We let $\mu_{\mathrm{hyp}}$ denote the hyperbolic metric on $M$, which is compatible with the
complex structure of $M$, and has constant negative curvature equal to minus one.
The hyperbolic line element $ds^{2}_{\hyp}$, resp.~the hyperbolic Laplacian
$\Delta_{\hyp}$, are given as
\begin{align*}
ds^{2}_{\hyp}:=\frac{dx^{2}+dy^{2}}{y^{2}},\quad\textrm{resp.}
\quad\Delta_{\hyp}:=-y^{2}\left(\frac{\partial^{2}}{\partial
x^{2}}+\frac{\partial^{2}}{\partial y^{2}}\right).
\end{align*}
By $d_{\mathrm{hyp}}(z,w)$ we denote the hyperbolic distance from $z\in\mathbb{H}$ to
$w\in\mathbb{H}$.


%


\subsection{Moonshine groups}

Let $N=p_1\cdots p_r$ be a square-free, non-negative integer.
The subset of $\SL_2(\R)$, defined by
\begin{align*}
  \Gamma_0(N)^+:=\left\{ e^{-1/2}\begin{pmatrix}a&b\\c&d\end{pmatrix}\in
    \SL_2(\R): \,\,\, ad-bc=e, \,\,\, a,b,c,d,e\in\Z, \,\,\, e\mid N,\ e\mid a,
    \ e\mid d,\ N\mid c \right\}
\end{align*}
is an arithmetic subgroup of $\SL_2(\R)$.  We use the terminology ``moonshine group''
of level $N$ to describe $\Gamma_0(N)^+$ because of the important role these groups
play in ``monstrous moonshine''.  Previously, the groups $\Gamma_0(N)^+$ were studied in
\cite{Hel66} where it was proved that if a subgroup $G\subseteq\SL_2(\R)$
is commensurable with $\SL_2(\Z)$, then there exists a square-free,
non-negative integer $N$ such that $G$ is a subgroup of $\Gamma_0(N)^+$.
We also refer to page 27 of \cite{Sh71} where the groups $\Gamma_0(N)^+$
are cited as examples of groups which are commensurable with $\SL_2(\Z)$
but non necessarily conjugate to a subgroup of $\SL_2(\Z)$.

Let $\{\pm \textrm{Id}\}$ denote the set of two elements consisting of the identity matrix $\textrm{Id}$ and its product with $-1$.
In general, if $\Gamma$ is a subgroup of $\SL_2(\R)$, we let $\overline{\Gamma} := \Gamma /\{\pm \textrm{Id}\}$ denote its
projection into $\textrm{PSL}_2(\R)$.

\subsection{Holomorphic Eisenstein series}

Following \cite{Se73}, we define a weakly modular form $f$ of weight $2k$ for $k \geq  1$ associated to $\Gamma$ to be a
function $f$ which is meromorphic on $\mathbb H$ and satisfies the transformation property
$$
f\left(\frac{az+b}{cz+d}\right) = (cz+d)^{-2k}f(z)
\,\,\,\,\,\textrm{for all $\begin{pmatrix}a&b\\c&d\end{pmatrix} \in \Gamma$.}
$$

Let $\Gamma$ be a Fuchsian group of the first kind that has at least one class of parabolic elements. By rescaling, if necessary, we may always
assume that the parabolic subgroup of $\Gamma$ has a fixed point at $\infty$, with identity scaling matrix. In this situation, any weakly modular
form $f$ will satisfy the relation $f(z+1)=f(z)$, so we can write
$$
f(z) = \sum\limits_{n=-\infty}^{\infty}a_{n}q_z^{n}
\,\,\,\,\,\textrm{where $q_z =e(z)= e^{2\pi iz}$.}
$$
If $a_{n} = 0$ for all $n < 0$, then $f$ is said to be holomorphic at the cusp at $\infty$.

A holomorphic modular form with respect to $\Gamma$ is a weakly modular form which is holomorphic on $\mathbb H$ and
in all of the cusps of $\Gamma$.  Examples of holomorphic modular forms are the holomorphic Eisenstein series, which are defined
as follows.  Let $\Gamma_{\infty}$ denote the subgroup of $\Gamma$ which stabilizes the cusp
at $\infty$.  For $k \geq 2$, let

\begin{equation} \label{E_2k, Gamma}
E_{2k,\Gamma}(z) := \sum_{\left(
    \begin{smallmatrix}
          * & * \\
          c & d \\
        \end{smallmatrix}
      \right) \in \Gamma_{\infty} \setminus \Gamma
}
(cz + d)^{-2k}.
\end{equation}
It is elementary to show that the series on the right-hand side of \eqref{E_2k, Gamma} is absolutely convergent for all integers
$k \geq 2$ and defines a holomorphic modular form of weight $2k$ with respect to $\Gamma$. Furthermore, the series $E_{2k, \Gamma}$ is bounded and non-vanishing
at cusps and such that
\begin{equation*}
E_{2k, \Gamma} (z) = 1 + O (\exp(-2\pi \Im (z))), \text{   as   } \Im (z) \to \infty.
\end{equation*}

When $\Gamma=\mathrm{PSL}_2(\Z)$, we denote $E_{2k, \mathrm{PSL}_2(\Z)}$ by $E_{2k}$.
The holomorphic forms $E_{2k}(z)$ have the $q-$expansions
$$
E_{2k}(z) = 1- \frac{4k}{B_{2k}} \sum_{n=1}^{\infty} \sigma_{2k-1}(n) q_z^n,
$$
where $B_{2k}$ denotes the $2k-$th Bernoulli number and $\sigma_l$ is the generalized divisor function, which
is defined by $\sigma_l(m) = \sum\limits_{d \mid m} d^l$.
By convention, we set $\sigma(m)=\sigma_1(m)$.

On the full modular surface, there is no weight $2$ holomorphic modular form.
Consider, however, the function $E_2(z)$ defined by its $q$-expansion
$$
E_2(z) = 1-24 \sum_{n=1}^{\infty} \sigma(n) q_z^n
$$
which transforms according to the formula
$$
E_2(\gamma z) = (cz+d)^2 E_2(z) + \frac{6}{\pi i}c (cz+d),
$$
for $\left(
        \begin{smallmatrix}
          * & * \\
          c & d \\
        \end{smallmatrix}
      \right) \in \textrm{PSL}_2(\mathbb{Z})$.
It is elementary to show that for a prime $p$, the function
\begin{equation} \label{E_2,p}
E_{2,p}(z) := E_2(z)-pE_2(pz)
\end{equation}
is a weight 2 holomorphic form associated to the congruence subgroup $\overline{\Gamma_0(p)}$ of $\textrm{PSL}_2(\mathbb{Z})$. The $q-$expansion
of $E_{2,p}$ is
\begin{equation}\label{q-exp E_2,p}
E_{2,p}(z)= (1-p) - 24\sum_{n=1}^{\infty}\sigma(n) (q_z^n - pq_z^{pn}).
\end{equation}

When $\Gamma = \overline{\Gamma_{0}^+(N)}$, we denote the forms $E_{2k, \overline{\Gamma_{0}^+(N)}}$ by $E_{2k}^{(N)}$.
In \cite{JST13} it is proved that $E_{2k}^{(N)}(z)$ may be expressed as a linear combination of forms $E_{2k}(z)$, with dilated arguments,
namely
\begin{align}\label{E_k, p proposit fla}
E_{2k}^{(N)}(z)= \frac1{\sigma_k(N)} \sum_{v \mid N}v^k E_{2k}(vz).
\end{align}





\subsection{Scattering matrices}

Assume that the surface $M$ has $p_{\Gamma}$ cusps, we let $P_{j}$ with $j=1,\ldots, p_{\Gamma}$ denote the
individual cusps.
Denote by $\phi_{jk}$, with $j,k=1, \ldots, p_{\Gamma}$, the entries of the hyperbolic scattering
matrix $\Phi_M(s)$ which are computed from the
constant terms in the Fourier expansion of the parabolic Eisenstein
series $\E^{\mathrm{par}}_{P_j}(z,s)$ associated to cusp $P_{j}$ in an expansion in the cusp $P_{k}$.
For all $j,k = 1,\ldots, p_{\Gamma}$, each function $\phi_{jk}$ has a simple pole at $s=1$ with residue
equal to $1/\vol_{\hyp}(M)$. Furthermore, $\phi_{jk}$ has a Laurent series expansion at $s=1$ which we write as
\begin{equation}\label{phi exp at s=1}
\phi_{jk}(s)= \frac{1}{\vol_{\hyp}(M) (s-1)} + \beta_{jk} + \gamma_{jk}(s-1) + O((s-1)^2), \text{  as  } s\to 1.
\end{equation}
After a slight renormalization and trivial generalization, Theorem 3-1 from \cite{Go73} asserts that the parabolic
Eisenstein series $\E^{\mathrm{par}}_{P_j}(z,s)$ admits the Laurent expansion
\begin{equation} \label{KronLimitPArGen}
\E^{\mathrm{par}}_{P_j}(z,s)= \frac{1}{\vol_{\hyp}(M) (s-1)} + \beta_{jj} - \frac{1}{\vol_{\hyp}(M)} \log \abs{\eta_{P_j}^4(z) \Im(z)} + f_j(z) (s-1) +  O((s-1)^2),
\end{equation}
as  $s \to 1$, for $j=1,\ldots, p_{\Gamma}$.

As the notation suggestions, the function $\eta_{P_j}(z)$ is a holomorphic form for $\Gamma$ and is a generalization of the classical eta function for the
full modular group. To be precise, $\eta_{P_j}(z)$ is an automorphic form corresponding to the multiplier system $v(\sigma)= \exp(i\pi S_{\Gamma,j}(\sigma))$,
where $S_{\Gamma,j}(\sigma)$ is a generalization of a Dedekind sum attached to a cusp $P_j$ for each $j=1,\ldots,p_{\Gamma}$ of $M$, meaning a real number
uniquely determined for every $\sigma = \left(
                                             \begin{smallmatrix}
                                               \ast & \ast \\
                                               c & d \\
                                             \end{smallmatrix}
                                           \right)\in \Gamma
$ which satisfies the relation
$$
\log\eta_{P_j}(\sigma(z))=\log\eta_{P_j}(z) + \frac{1}{2} \log (cz+d) + \pi i S_{\Gamma,j}(\sigma).
$$

The coefficient $f_j(z)$ multiplying $(s-1)$ in formula \eqref{KronLimitPArGen}
is a certain function, whose behavior is not of interest to us in this paper. This term would
probably yield to a definition of generalized Dedekind sums; see, for example, \cite{Ta86}.

Finally, let us set the notation
\begin{equation}\label{phi exp at s=0}
\phi_{jk}(s)= a_{jk} + b_{jk}s + c_{jk}s^2 + O(s^3) \,\,\,\,\,\textrm{as $s \rightarrow 0$}
\end{equation}
for the coefficients in the Laurent expansion of $\phi_{jk}$ near $s=0$.
Note that the form of this expansion is justified by the assumption made
in subsection \ref{subsection_assumption}.


\section{Kronecker's limit formula for parabolic Eisenstein series}\label{sec: Kron limir parabolic}

\vskip .10in
In this section we will re-write the Kronecker limit formula for the parabolic Eisenstein series as an expression involving
the Laurent expansion near $s=0$.  We begin with the following lemma which states certain relations amongst coefficients
appearing in \eqref{phi exp at s=1} and \eqref{phi exp at s=0}.  \it To repeat, we assume that each parabolic
Eisenstein series ${\cal E}^{\mathrm{par}}_{P_j}(z,s)$ is holomorphic at $s=0$. \rm

\begin{lemma}
With the notation in \eqref{phi exp at s=1} and \eqref{phi exp at s=0}, we have, for each $k, l = 1,\ldots,p_{\Gamma}$, the
following relations:
\begin{equation} \label{sum a_jk}
\sum_{j=1}^{p_{\Gamma}} a_{jk} = 0,
\end{equation}
\begin{equation} \label{sum with b_jk}
\sum_{j=1}^{p_{\Gamma}}\left( - \frac{b_{jk}}{\vol_{\hyp}(M)} + a_{jk}\beta_{jl}\right) = \delta_{kl},
\end{equation}
\begin{equation}\label{sum with c_jk}
\sum_{j=1}^{p_{\Gamma}}\left(- \frac{c_{jk}}{\vol_{\hyp}(M)} +b_{jk}\beta_{jl}\right) = \sum_{j=1}^{p_{\Gamma}} a_{jk}\gamma_{jl},
\end{equation}
where $\delta_{kl}$ is the Kronecker symbol.
\end{lemma}
\begin{proof}
The relations \eqref{sum a_jk} through \eqref{sum with c_jk} are immediate consequences of the functional equation for the scattering determinant,
namely the formula $\Phi_M(s)\Phi_M(1-s) = \textrm{Id}$.  In particular, the formulae are obtained by computing the
coefficients of $s^{-1}$, $1$, and $s$ in the Laurent expansion near $s=0$.
\end{proof}

\vskip .10in
\begin{proposition}\label{prop: Kronecker limit as s to 0}
With the notation in \eqref{phi exp at s=1} and \eqref{phi exp at s=0}, the parabolic Eisenstein series
$\E^{\mathrm{par}}_{P_j}(z,s)$ has a Taylor series expansion at $s=0$ which can be written as
\begin{multline} \label{parabolic Kron limit as s to 0}
\E^{\mathrm{par}}_{P_j}(z,s) = \sum_{k=1}^{p_{\Gamma}} \left[ - \frac{b_{jk}}{\vol_{\hyp}(M)} + a_{jk}\left( \beta_{kk} - \frac{1}{\vol_{\hyp}(M)}
\log \left|\eta_{P_k}^4(z)\Im z \right|\right) \right] +\\
+ s \cdot \sum_{k=1}^{p_{\Gamma}} \left[ - \frac{c_{jk}}{\vol_{\hyp}(M)} + b_{jk}\left( \beta_{kk} - \frac{1}{\vol_{\hyp}(M)} \log \left|\eta_{P_k}^4(z)\Im z \right|\right)
+a_{jk}f_k(z) \right] + O(s^2).
\end{multline}
\end{proposition}
\begin{proof}
The result is a straightforward computation based on the functional equation
$$
(\E^{\mathrm{par}}_{1}(z,s)\,\, ....\,\, \E^{\mathrm{par}}_{p}(z,s))^{T} = \Phi_M(s)
(\E^{\mathrm{par}}_{1}(z,1-s)\,\, ....\,\, \E^{\mathrm{par}}_{p}(z,1-s))^{T}
$$
together with the expansions \eqref{KronLimitPArGen} and \eqref{phi exp at s=0}.
\end{proof}

In the case when $p_{\Gamma}=1$, the relations \eqref{sum a_jk} through \eqref{sum with c_jk} and Proposition \ref{prop: Kronecker limit as s to 0}
become particularly simple and yield an elegant statement.  As is standard, the cusp is normalized to be at $\infty$, and
the associated Eisenstein series, eta function, scattering coefficients, etc.~are written with the subscript $\infty$.

\begin{corollary} \label{Kron limit as s to 0, one cusp}
The Kronecker limit formula for parabolic Eisenstein series $\E^{\mathrm{par}}_{\infty}$ on a finite volume Riemann surface with one cusp at $\infty$
can be written as
\begin{equation} \label{KronLimas s to 0}
\E^{\mathrm{par}}_{\infty}(z,s)= 1+ \log (\abs{\eta_{\infty}^4(z)} \Im(z))s +  O(s^2), \text{  as  } s \to 0.
\end{equation}
\end{corollary}

\vskip .06in
\begin{example}\rm \label{ex: moonshine groups}
In the case when $\Gamma=\overline{\Gamma_0^+(N)}$, for a square-free, positive integer $N$, the quotient space
$X_N:=\overline{\Gamma_0^+(N)} \backslash \h$ has one cusp. The automorphic form $\eta_{\infty}$ is explicitly computed in \cite{JST13}, where it is proved that
\begin{align*}
\eta_{\infty}(z) = \sqrt[2^r]{\prod_{v \mid N}  \eta(vz)}.
\end{align*}
\end{example}

\vskip .06in
\begin{example}\rm \label{ex: congruence subgr}
In the case when $\Gamma$ is the group $\overline{\Gamma_0(N)}$, for a positive integer $N$,
the corresponding quotient space $M_N:=\overline{\Gamma_{0}(N)}\backslash \h$ has many cusps.
Using a standard fundamental domain, $M_{N}$ has cusps at $\infty$, at $0$ and, in the case when $N$ is not prime, at the rational points $1/v$,
where $v \mid N$ is such that $(v, \frac{N}{v}) =1$, where $(\cdot,\, \cdot)$ stands for the greatest common divisor.  As in
the above example, let use the subscript $\infty$ to denote data associated to the cusp at $\infty$.  In particular,
the automorphic form $\eta_{\infty}$ in the example under consideration was explicitly computed in \cite{Vassileva96}, where it is proved that
$$
\eta_{\infty}(z)= \sqrt[\varphi(N)]{\prod_{v \mid N}  \eta(vz) ^{v \mu(N/v)}},
$$
where $\varphi(N)$ is the Euler $\varphi-$function and $\mu$ denotes the M\"obius function.
In the case of other cusps $P_k$, the automorphic form $\eta_{P_k}$ was also computed in \cite{Vassileva96},
but the expressions are more involved so we omit repeating the formulas here.

Also, for the cusp at $\infty$ and the principal congruence subgroup $\Gamma(N)$, the eta-function is computed in Theorem
1, page 405 of \cite{Ta86}.
\end{example}

\vskip .10in
\section{Kronecker's limit formula for elliptic Eisenstein series}

The function $H_{\Gamma}(z,w)$, defined in (\ref{Kronecker_elliptic}) is called
the \textit{elliptic Kronecker limit function at $w$}. It satisfies the transformation rule
\begin{align}\label{H e j transf. rule}
H_{\Gamma}(\gamma z, w) = \eps_{w}(\gamma) (cz + d)^{2C_{w}} H_{\Gamma}(z,w), \text{  for any  } \gamma = \begin{pmatrix} * & * \\ c & d \end{pmatrix} \in \Gamma,
\end{align}
where  $\eps_{w}(\gamma) \in \C$ is a constant of absolute value $1$,
independent of $z$ and
\begin{equation}\label{C_w}
C_w= 2\pi /(\mathrm{ord}(w) \vol_{\hyp}(M)),
\end{equation}
see \cite{vP10}, Proposition 6.1.2., or \cite{vP15}.
Since $H_{\Gamma}(z,w)$, as a function of $z$, is finite and non-zero at the cusp
$P_{1} = \infty$, we may re-scale the function and
assume, without lost of generality, that $H_{\Gamma}(z,w)$ is real at the cusp $\infty$.

\vskip .10in
We begin by studying the asymptotic behavior of $H_{\Gamma}(\sigma_{P_l}z,w)$
as $y=\Im(z) \to\infty$, for $l=1, \ldots, p_{\Gamma}$.

\begin{proposition} \label{prop: behavior of H(z,w)}
For any cusp $P_l$, with $l=1,\ldots,p_{\Gamma}$, let
\begin{equation} \label{B_e_j}
B_{w,P_l}=-C_{w} \left( 2-\log 2 + \log \left| \eta_{P_l}^4(w) \Im(w) \right| - \beta_{ll}\vol_{\hyp}(M)\right).
\end{equation}
Then there exists a constant $a_{w, P_l} \in \C$ of modulus one such that
\begin{equation*}
H(\sigma_{P_l}z, w) = a_{w,P_l} \exp(-B_{w,P_l}) |c_l z + d_l|^{2C_w} + O(\exp(-2\pi \Im(z))),
\text{ as } \Im(z)\rightarrow \infty\,,
\end{equation*}
where $\sigma_{P_l} =\left( \begin{smallmatrix} * & * \\ c_l & d_l \end{smallmatrix}\right)$ is the scaling matrix for a cusp $P_l$ and $C_w$ is defined by \eqref{C_w}.
\end{proposition}

\begin{proof}
The proof closely follows the proof of \cite{vP10},
Proposition 6.2.2. when combined with the Taylor series
expansion \eqref{parabolic Kron limit as s to 0} of the parabolic Eisenstein series at $s=0$.
For the convenience of the reader, we now present the complete argument.

Combining the equation \eqref{Kronecker_elliptic} with
the proof of Proposition 6.1.1 from \cite{vP10}, taking $e_j=w$, we can write
$$
-\log(\vert{H_{\Gamma}(z, w)} \vert \Im(z)^{C_{w}}) = \mathcal{K}_{w} (z),
$$
where the function $\mathcal{K}_{w} (z)$ can be expressed as the sum of two
terms: A term $\F_{w} (z)$ arising from the spectral expansion and
a term $\mathcal{G}_{w} (z)$ which can be expressed as the sum over the
group.  Furthermore, for $z\in\h$ such that $\Im z > \Im (\gamma w)$ for all $\gamma \in \Gamma$
the parabolic Fourier expansion of $\mathcal{K}_{w}(\sigma_{P_l} z)$ is given by
$$
\mathcal{K}_{w}(\sigma_{P_l} z) = \sum_{m\in\Z} b_{m,w,P_l}(y)e(mx)
$$
with coefficients $b_{m,w,P_l}(y)$ given by
$$
b_{m,w,P_l}(y)=\int\limits_{0}^{1}\mathcal{K}_{w}(\sigma_{P_l} z)e(-mx).
$$
Since the hyperbolic Laplacian is $\mathrm{SL}_2-$invariant, we easily generalize computations from p. 128 of \cite{vP10}
to deduce that
$$
\mathcal{K}_{w} (\sigma_{P_l}z) = -C_{w}\log y + A_{w,P_l}y + B_{w,P_l}
+ \sum_{m=1}^\infty (A_{m; w,P_l} e(mz)
+ \overline{A}_{m; w,P_l} e(-m\overline{z})),
$$
for some constants $A_{w, P_l}, B_{w,P_l} \in \R$ and complex constants $A_{m;w,P_l}$.

Let us introduce the notation
\begin{align}\label{f e j}
f_{w,P_l}(z):=\exp\left(-2\sum_{m=1}^\infty A_{m; w,P_l} e(mz) \right),
\end{align}
from which one immediately can write
\begin{align}\label{K e j main formula}
\mathcal{K}_{w} (\sigma_{P_l} z) = A_{w,P_l}y + B_{w,P_l} - \log(\abs{f_{w,P_l}(z)} \Im(z)^{C_{w}}).
\end{align}
When employing \eqref{K e j main formula}, we can re-write \eqref{Kronecker_elliptic} as
\begin{align}\label{Fla for comparison}
\E^{\mathrm{ell}}_{w}(\sigma_{P_l} z,s)
&- h_{w}(s)\sum_{j=1}^{p_{\Gamma}} \E^{\mathrm{par}}_{P_j}(w,1-s)
\E^{\mathrm{par}}_{P_j}(\sigma_{P_l} z,s) = \\ &-C_{w}
+ (A_{w,P_l}y + B_{w,P_l} - \log(\vert{f_{w, P_l}(z)}\vert \Im(z)^{C_{w}})) \cdot s + O(s^2) \nonumber,
\end{align}
as $s \rightarrow 0$, where
\begin{equation} \label{h_w}
h_{w}(s):= \frac{2^s \sqrt{\pi} \,\Gamma(s-1/2)}{\mathrm{ord}(w)\Gamma(s)}.
\end{equation}
As in \cite{vP10}, pp. 129--130, we use the functional
equation of the parabolic Eisenstein series and consider the constant term
in the Fourier series expansion, as a function of $z$, of the function
\begin{equation} \label{const term}
\E^{\mathrm{ell}}_{w}(\sigma_{P_l} z,s)
- h_{w}(s) \sum_{j=1}^{p_{\Gamma}} \E^{\mathrm{par}}_{P_j}(w,1-s)
\E^{\mathrm{par}}_{P_j}(\sigma_{P_l} z,s)=\E^{\mathrm{ell}}_{w}(\sigma_{P_l} z,s)
- h_{w}(s) \sum_{j=1}^{p_{\Gamma}} \E^{\mathrm{par}}_{P_j}(w,s)
\E^{\mathrm{par}}_{P_j}(\sigma_{P_l} z,1-s).
\end{equation}
The constant term is given by
$$
-h_{w}(s)\sum_{j=1}^{p_{\Gamma}} \phi_{jl}(1-s)y^s\E^{\mathrm{par}}_{P_j}(w,s)= -\frac{\sqrt{\pi}}{\mathrm{ord}(w)}
\frac{\Gamma(s-1/2)}{\Gamma(s)}(2y)^s \sum_{j=1}^{p_{\Gamma}} \phi_{jl}(1-s)\E^{\mathrm{par}}_{P_j}(w,s).
$$
Recall the expansions
\begin{equation} \label{gamma s-1/2}
\Gamma(s-1/2)= -2\sqrt{\pi}\left(1+(2-\gamma-2\log 2) s + O(s^2)\right),
\end{equation}
\begin{equation} \label{gamma s}
\frac{1}{\Gamma(s)}= s \left(1+ \gamma s + O(s^2)\right), \text{    and     }(2y)^s= 1+ s\log(2y) + O(s^2),
\end{equation}
which hold when $s\to 0$, where, as usual, $\gamma$ denotes the Euler constant.
When combining these expressions with \eqref{phi exp at s=1},
we can write the asymptotic expansions near $s=0$ of the constant term in the Fourier series expansion of \eqref{const term} as
\begin{equation}\label{F.series const intermediate}
\frac{2\pi}{\mathrm{ord}(w)} \left(1+ (2+\log y -\log 2)s + O(s^2)\right) \cdot \sum_{j=1}^{p_{\Gamma}} \left( -\frac{1}{\vol_{\hyp}(M)} + \beta_{jl}s + O(s^2) \right)\E^{\mathrm{par}}_{P_j}(w,s).
\end{equation}
Let us now compute the first two terms in the Taylor series expansion at $s=0$ of the expression
\begin{equation}\label{F.series const intermediate 2}
\sum_{j=1}^{p_{\Gamma}} \left( -\frac{1}{\vol_{\hyp}(M)}+ \beta_{jl}s + O(s^2) \right)\E^{\mathrm{par}}_{P_j}(w,s).
\end{equation}
By applying \eqref{parabolic Kron limit as s to 0}, we conclude that the constant term in the Taylor series expansion of (\ref{F.series const intermediate 2})
is
$$
\sum_{j=1}^{p_{\Gamma}} \sum_{k=1}^{p_{\Gamma}} \frac{-1}{\vol_{\hyp}(M)} \left( -\frac{b_{jk}}{\vol_{\hyp}(M)} +a_{jk} \beta_{kk}  -\frac{a_{jk}}{\vol_{\hyp}(M)} \log \left| \eta_{P_k}^4(w) \Im(w) \right|\right).
$$
Applying relations \eqref{sum a_jk} and \eqref{sum with b_jk} we then obtain, by manipulation of the sums,
that the constant term in (\ref{F.series const intermediate 2}) is equal to
$\displaystyle -1/\vol_{\hyp}(M)$. The factor multiplying $s$ is equal to
\begin{multline*}
\sum_{j=1}^{p_{\Gamma}} \sum_{k=1}^{p_{\Gamma}} \frac{-1}{\vol_{\hyp}(M)} \left( -\frac{c_{jk}}{\vol_{\hyp}(M)} +b_{jk} \beta_{kk}  -\frac{b_{jk}}{\vol_{\hyp}(M)} \log \left| \eta_{P_k}^4(w) \Im(w) \right| + a_{jk} f_k(w)\right) \\
+\sum_{j=1}^{p_{\Gamma}} \sum_{k=1}^{p_{\Gamma}} \beta_{jl} \left( -\frac{b_{jk}}{\vol_{\hyp}(M)} +a_{jk} \beta_{kk}  -\frac{a_{jk}}{\vol_{\hyp}(M)} \log \left| \eta_{P_k}^4(w) \Im(w) \right|\right).
\end{multline*}
Applying relations \eqref{sum a_jk} to \eqref{sum with c_jk} we get that
$$
\sum_{j=1}^{p_{\Gamma}} \sum_{k=1}^{p_{\Gamma}}a_{jk}f_k(w)=0
$$
and
\begin{align*}
\sum_{k=1}^{p_{\Gamma}}&\left( \frac{-1}{\vol_{\hyp}(M)} \log \left| \eta_{P_k}^4(w) \Im(w) \right| +
\beta_{kk} \right) \sum_{j=1}^{p_{\Gamma}} \left( -\frac{b_{jk}}{\vol_{\hyp}(M)} + a_{jk} \beta_{jl}\right)
\\&= \frac{-1}{\vol_{\hyp}(M)} \log \left| \eta_{P_l}^4(w) \Im(w) \right| + \beta_{ll}
\end{align*}
as well as
$$
\sum_{j=1}^{p_{\Gamma}} \sum_{k=1}^{p_{\Gamma}} \left( \frac{-c_{jk}}{\vol_{\hyp}(M)} +
b_{jk} \beta_{jl}\right)= \sum_{j=1}^{p_{\Gamma}} \sum_{k=1}^{p_{\Gamma}} a_{jk}\gamma_{jl} =0.
$$
Therefore, the factor multiplying $s$ in the Taylor series expansion of \eqref{F.series const intermediate 2} is equal to
$$
\frac{-1}{\vol_{\hyp}(M)} \log \left| \eta_{P_l}^4(w) \Im(w) \right| + \beta_{ll}.
$$
Inserting this into \eqref{F.series const intermediate} we see that the constant term in the Fourier series expansion of \eqref{const term} is given by
$$
-C_w-C_w\left( 2-\log 2 + \log y + \log \left| \eta_{p_l}^4(w) \Im(w) \right| - \beta_{ll}\vol_{\hyp}(M) \right)s + O(s^2),
$$
as $s \to 0$.
Comparing this result with the right-hand side of formula
\eqref{Fla for comparison}, having in mind the definition of the number $C_{w}$, we immediately deduce that $A_{w, P_l}=0$,
$$
B_{w,P_l}=-C_{w} \left( 2-\log 2 + \log \left| \eta_{P_l}^4(w) \Im(w) \right| - \beta_{ll}\vol_{\hyp}(M)\right)
$$
and
$$
\mathcal{K}_{w} (\sigma_{P_l} z) = -\log(\vert{H_{\Gamma}(\sigma_{P_l} z, w)}\vert |c_l z +d_l |^{-2C_w} \Im(z)^{C_{w}})
= B_{w,P_l} - \log(\vert{f_{w,P_l}(z)}\vert \Im(z)^{C_{w}} ),
$$
where the function $f_{w,P_l}$ is defined by \eqref{f e j}.
From \eqref{f e j} we deduce that
$$
\abs{f_{w,P_l}(z)} = \exp \left( -2 \Re \left( \sum_{m=1}^{\infty} A_{m;w, P_l} e(mz) \right)\right) = 1 + O(\exp(-2\pi \Im (z))),
$$
as $\Im (z) \to \infty$.  Therefore,
$$
\abs{H_{\Gamma}(\sigma_{P_l} z, w)} = \exp(-B_{w,P_l})|c_l z +d_l |^{2C_w}  +  O(\exp(-2\pi \Im (z))), \text{  as  } \Im(z) \to \infty\,,
$$
and the proof is complete.
\end{proof}

\vskip .06in
\begin{example}{\bf Moonshine groups.}\rm \label{ex: constants B_N}
Let $N=p_1 \cdot \ldots \cdot p_r$ be a squarefree number. 
Let $X_N= \overline{\Gamma_0(N)^+} \setminus \h$.
The surface $X_N$ possesses one cusp at $\infty$ with identity scaling matrix.
The scattering determinant $\varphi_N$ associated to the only cusp of $X_N$ at $\infty$ is computed in \cite{JST12}, where
it was shown that
$$
\varphi_N(s)=\sqrt{\pi}\frac{\Gamma(s-1/2)}{\Gamma(s)}\frac{\zeta(2s-1)}{\zeta(2s)}\cdot D_N(s),
$$
where $\zeta(s)$ is the Riemann zeta function and
$$
D_N(s)=\prod_{j=1}^r\frac{p_j^{1-s}+1}{p_j^s+1}= \frac{1}{N^{s-1}}\prod_{j=1}^r\frac{p_j^{s-1}+1}{p_j^s+1}.
$$
Let $b_{N}$ denote the constant term in the Laurent series expansion of $\varphi_N(s)$ at $s=1$. One can compute $b_N$ by expanding functions
$D_N(s)$, $\Gamma(s)$ and $\zeta(s)$ in their Laurent expansions at $s=1$, which would yield the expressions
$$
D_N(s)= \frac{2^r}{\sigma(N)}\left(1 + (s-1) \left(\sum_{j=1}^{r} \frac{(1-p_j)\log p_j}{2(p_j+1)} - \log N\right) + O((s-1)^2) \right),
$$
and
\begin{align}
\sqrt{\pi}\frac{\Gamma(s-1/2)}{\Gamma(s)}
&= \pi \left(1-2\log 2 (s-1)
+ O((s-1)^2)\right), \label{gamma exp}
\end{align}
as well as
\begin{align}
\frac{\zeta(2s-1)}{\zeta(2s)}
&= \frac{6}{\pi^2} \left( \frac{1}{2(s-1)} - \log (2\pi) + 1-12\zeta'(-1) + O(s-1)\right).
\label{zeta exp}
\end{align}
Multiplying expansions \eqref{gamma exp} and \eqref{zeta exp} and using that
$$
\frac{1}{\vol_{\hyp} (X_N)} = \frac{3 \cdot 2^r}{\pi \,\sigma(N)},
$$
which was proved in \cite{JST13}, we arrive at the expression
\begin{equation} \label{b_N}
b_N= - \frac{1}{\vol_{\hyp} (X_N) }\left( \sum_{j=1}^{r} \frac{(p_j -1)\log p_j}{2(p_j+1)}- \log N + 2\log (4\pi) + 24\zeta'(-1) - 2\right).
\end{equation}
With this formula, Proposition \ref{prop: behavior of H(z,w)}, and Example \ref{ex: moonshine groups} we conclude that the elliptic Kronecker limit function
$H_N(z,w) := H_{\overline{\Gamma_0^+(N)}} (z,w)$ associated to the point $w \in X_N$ may we written as
$$
H_N(z,w)= a_{N,w}\exp(-B_{N,w}) + \exp(-2\pi\Im (z)), \text{   as   } \Im (z) \to \infty,
$$
where $a_{N,w}$ is a complex constant of modulus one and
\begin{align*}
B_{N,w} &= - \frac{2\pi}{\mathrm{ord}(w)\vol_{\hyp} (X_N)}\left( \sum_{j=1}^{r} \frac{(p_j -1)\log p_j}{2(p_j+1)}-\log N
+ C+ \log \left(\sqrt[2^r]{\prod_{v \mid N}
\abs{\eta(v w)}^4} \cdot \Im (w)\right) \right)\notag
\end{align*}
with $C:=\log (8\pi^2) + 24\zeta'(-1)$.
\end{example}

\vskip .06in
\begin{example}{\bf Congruence subgroups of prime level.} \rm
Let $M_p= \overline{\Gamma_0(p)}\setminus \h$, where $p$ is a prime. The surface $M_p$ has two cusps,
at $\infty$ and $0$. The scaling matrix for the cusp at $\infty$ is identity matrix. The scattering matrix in
this setting is computed in \cite{He83} and is given by
$$
\Phi_{M_p}(s)= \sqrt{\pi} \frac{\Gamma(s-1/2)}{\Gamma(s)} \frac{\zeta(2s-1)}{\zeta(2s)} \cdot \frac{1}{p^{2s}-1} \left(
                           \begin{array}{cc}
                             p-1 & p^s-p^{1-s} \\
                             p^s-p^{1-s} & p-1 \\
                           \end{array}
                         \right).
$$
Using the expansions \eqref{gamma exp} and \eqref{zeta exp}, together with $\vol_{\hyp}(M_p)=\pi(p+1)/3$
and the expansion
$$
\frac{p-1}{p^{2s}-1}= \frac{1}{p+1}-\frac{2p^2 \log p}{(p-1)(p+1)^2} (s-1) + O((s-1)^2) \text{  as  } s \to 1\,,
$$
we conclude that the
coefficients $\beta_{11}$ and $\beta_{22}$ in the Laurent series
expansion \eqref{phi exp at s=1} are given by
$$
\beta_{11}=\beta_{22}= -\frac{2}{\vol_{\hyp} (M_p)}\left( \log (4\pi p) + 12\zeta'(-1) -1 + \frac{\log p}{p^2 -1} \right).
$$
Therefore, from Proposition \ref{prop: behavior of H(z,w)}, when applied to the cusp at $\infty$,
and Example \ref{ex: congruence subgr},
we conclude that the elliptic Kronecker limit function $\widetilde{H}_p(z,w) := H_{\overline{\Gamma_0(p)}} (z,w)$ associated to the point $w \in M_p$
can be written as
$$
\widetilde{H}_p(z,w)= \widetilde{a}_{p,w}\exp(-\widetilde{B}_{p,w}) + \exp(-2\pi\Im (z)), \text{   as   } \Im (z) \to \infty,
$$
where $\widetilde{a}_{p,w}$ is a complex constant of modulus one and
\begin{align*}
\widetilde{B}_{p,w} &= - \frac{2\pi}{\mathrm{ord}(w)\vol_{\hyp} (M_p)}\left( \frac{2 p^2 \log p }{p^2-1}
+C
+\log \left(\abs{\sqrt[p-1]{\frac{\eta(p w) ^p}{\eta(w)}} \cdot \Im (w)}\right) \right)
\end{align*}
with $C:=\log (8\pi^2) + 24\zeta'(-1)$.
\end{example}

\vskip .10in


%

\section{A factorization theorem}

In (\ref{elliptic_at_i}) and (\ref{elliptic_at_rho}) one has an evaluation of the elliptic Kronecker limit function
in the special case when $\Gamma = \mathrm{PSL}_2(\mathbb{Z})$
and $w=i$ or $w=\rho= \exp(2\pi i /3)$ are the elliptic fixed points of $\mathrm{PSL}_2(\mathbb{Z})$.  The following
theorem generalizes these results.

\begin{theorem} \label{thm: factorization}
Let $M = \Gamma\setminus \h$ be a finite volume Riemann surface with at least one cusp, which we assume to be at $\infty$ with identity scaling matrix.
Let $k$ be a fixed positive integer such that there exists a weight $2k$ holomorphic form $f_{2k}$ on $M$ which is non-vanishing in all cusps and with
$q-$expansion at $\infty$ given by
\begin{equation} \label{q exp. of f_2k}
f_{2k}(z)= b_{f_{2k}} + \sum_{n=1}^{\infty}b_{f_{2k}}(n)q_z^n.
\end{equation}
Let $Z(f_{2k})$ denote the set of all zeros $f_{2k}$ counted according to their multiplicities and let us define the function
$$
H_{f_{2k}}(z):= \prod_{w \in Z(f_{2k})} H_{\Gamma}(z,w),
$$
where, as above, $H_{\Gamma}(z,w)$ is the elliptic Kronecker limit function.  Then there exists a complex constant $c_{f_{2k}}$ such that
\begin{equation} \label{factorization fla}
f_{2k}(z) = c_{f_{2k}}H_{f_{2k}}(z)
\end{equation}
and
$$
\abs{c_{f_{2k}}} =\abs{b_{f_{2k}}} \exp \left( \sum_{w\in Z(f_{2k})} B_{w, \infty} \right ),
$$
where  $B_{w,\infty}$ is defined in \eqref{B_e_j}.
\end{theorem}
\begin{proof}
Assume that $f_{2k}$ possesses $m+l\geq 1$ zeroes on $M$, where $m$ zeros are at the elliptic points
$e_j$ of $M$, $j=1,\ldots,m$,
and $l$ zeroes are at the non-elliptic points $w_i \in M$; of course, all zeroes are counted with multiplicities.
Then $H_{f_{2k}}(z)$ is a holomorphic function on $M$ which is vanishing if and only if $z \in Z(f_{2k})$ and which according to \eqref{H e j transf. rule}
satisfies the transformation rule
$$
H_{f_{2k}}(\gamma z) = \eps_{f_{2k}}(\gamma)(cz+d)^{C_{f_{2k}}} H_{f_{2k}}(z), \text{  for any  } \gamma = \begin{pmatrix} * & * \\ c & d \end{pmatrix} \in \Gamma,
$$
where $\eps_{f_{2k}}(\gamma)$ is a constant of modulus one and
$$
C_{f_{2k}} = \frac{4\pi}{\vol_{\hyp} (M)} \left(\sum_{j=1}^{m} \frac{1}{n_{e_j}} + l \right).
$$

The classical Riemann-Roch theorem relates the number of zeros of a holomorphic form to its weight and the genus of
$M$ in the case $M$ is smooth and compact.  A generalization of the relation follows from Proposition 7, page II-7,
of \cite{SCM66} which, in the case under consideration, yields the formula
\begin{align} \label{zeros f-la}
k \cdot \frac{\vol_{\hyp}(M)}{2\pi}= \sum_{e \in \mathcal{E}_N} \frac{1}{n_e} v_{e}(f) + \sum_{z\in M \setminus \mathcal{E}_N} v_{z}(f),
\end{align}
where $\mathcal{E}_N$ denotes the set of elliptic points in $M$, $n_e$ is the order of the elliptic point $e\in \mathcal{E}_N$ and $v_z(f)$ denotes the order of the zero $z$ of $f$.

Since $Z(f_{2k})$ is the set of all vanishing points of $f_{2k}$, formula \eqref{zeros f-la} implies that
$$
2k \cdot \frac{\vol_{\hyp} (M)}{4\pi} = \sum_{j=1}^{m} \frac{1}{n_{e_j}} + l,
$$
hence $C_{f_{2k}} = 2k$. In other words, $H_{f_{2k}}(z)$ is a holomorphic function on $M$, vanishing if
and only if $z \in Z(f_{2k})$ and satisfying transformation rule
$$
H_{f_{2k}}(\gamma z) = \eps_{f_{2k}}(\gamma)(cz+d)^{2k} H_{f_{2k}}(z), \text{  for any  } \gamma = \begin{pmatrix} * & * \\ c & d \end{pmatrix} \in \Gamma.
$$
By Proposition \ref{prop: behavior of H(z,w)}, we have that for any $w \in Z(f_{2k})$ and any cusp $P_l$ of $M$, with $l=1,\ldots,p_{\Gamma}$,
the function
$$
F_{f_{2k}}(z):= \frac{H_{f_{2k}}(z)}{f_{2k}(z)}
$$
is a non-vanishing holomorphic function on $M$, bounded and non-zero at the cusp at $\infty$ and has at most polynomial growth in any other cusp of $M$.
Therefore, the function $\log \vert F_{f_{2k}}(z)\vert$ is harmonic on $M$ whose growth in any cusp is such that $\log \vert F_{f_{2k}}(z)\vert$
is $L^{2}$ on $M$.  As a result, $\log \vert F_{f_{2k}}(z)\vert$ admits a spectral expansion; see \cite{He83} or \cite{Iwa02}.  Since $\log \vert F_{f_{2k}}(z)\vert$
is harmonic, one can use integration by parts to show that $\log \vert F_{f_{2k}}(z)\vert$ is orthogonal to any eigenfunction of the Laplacian.
Therefore, from the spectral expansion, one concludes that $\log \vert F_{f_{2k}}(z)\vert$ is constant, hence so is $F_{f_{2k}}(z)$.  The
evaluation of the constant is obtained by considering the limiting behavior as $z$ approaches $\infty$. With all this, the proof of
\eqref{factorization fla} is complete.
\end{proof}

\section{Examples of factorization}

\subsection{An arbitrary surface with one cusp}

In the case when a surface $M$ has one cusp, we get the following special case of Theorem \ref{thm: factorization}.

\begin{corollary} \label{cor:factorization, one cusp}
Let $M = \Gamma \setminus \h$ be a finite volume Riemann surface with one cusp,
which we assume to be at $\infty$ with identity scaling matrix. Then the weight $2k$ holomorphic Eisenstein series
$E_{2k, \Gamma}$ defined in \eqref{E_2k, Gamma}
can be represented as
$$
E_{2k, \Gamma}(z) = a_{E_{2k, \Gamma}} B_{E_{2k, \Gamma}}\prod_{w \in Z(E_{2k, \Gamma})}H_{\Gamma}(z,w),
$$
where $a_{E_{2k, \Gamma}}$ is a complex constant of modulus one and
$$
B_{E_{2k, \Gamma}} = \prod_{w \in Z(E_{2k, \Gamma})} \exp \left(C_w \left(\log 2 -2 + \beta_M \vol_{\hyp}(M) \right)\right) \cdot \left|\eta_{\infty}^4(w) \Im (w) \right|^{-C_w} .
$$
As before, $\eta_{\infty}$ is the parabolic Kronecker limit function defined in section \ref{sec: Kron limir parabolic}, formula \eqref{KronLimitPArGen}, and
$\beta_M$ is the constant term in the Laurent series expansion of the scattering determinant on $M$.
\end{corollary}

In this case, due to a very simple form of the Kronecker limit formula for parabolic Eisenstein
series as $s\to 0$, the factorization theorem yields an interesting form of the Kronecker limit formula
for elliptic Eisenstein series, which we state as the following proposition.
\begin{proposition} \label{prop:Ell Kron limit one cusp}
Let $M = \Gamma \setminus \h$ be a finite volume Riemann surface with one cusp,
which we assume to be at $\infty$ with identity scaling matrix. Let $k$ be a fixed positive integer such that there exists a weight $2k$ holomorphic form $f_{2k}$ on $M$ with $q-$expansion at $\infty$ given by \eqref{q exp. of f_2k}. Then
\begin{equation} \label{ell kroneck limit one cusp}
\sum_{w\in Z(f_{2k})} \mathcal{E}^{\mathrm{ell}}_{w}(z,s)=
-s\log\left( |f_{2k}(z)| |\eta_{\infty} ^4(z)|^{-k}\right) + s\log|b_{f_{2k}}| + O(s^2)
\end{equation}
as $s\to 0$, where  $Z(f_{2k})$ denotes the set of all zeros of $f_{2k}$ counted with multiplicities.
\end{proposition}

\begin{proof}
We start with formula \eqref{Kronecker_elliptic}, which we divide by $\mathrm{ord}(w)$, and take the sum over all $w \in Z(f_{2k})$ to get
\begin{align} \label{Kronecker limit ell 2}
\sum_{w\in Z(f_{2k})} \mathcal{E}^{\mathrm{ell}}_{w}(z,s)&- \mathcal{E}^{\mathrm{par}}_{\infty}(z,s) \sum_{w\in Z(f_{2k})} h_w(s)
\mathcal{E}^{\mathrm{par}}_{\infty}(w,1-s) = \notag \\
&-\sum_{w\in Z(f_{2k})} C_w \left( 1 + s \log(\Im z)\right) - \log\left( \prod_{w\in Z(f_{2k})} |H_{\Gamma}(z,w)| \right)\cdot s
+O(s^2)
\end{align}
as $s\to 0$, where $C_w$ and $h_w$ are defined by \eqref{C_w} and \eqref{h_w} respectively.
One now expands the second term on the left hand side of \eqref{Kronecker limit ell 2} into a Taylor series at $s=0$ by applying formulas \eqref{gamma s-1/2}, \eqref{gamma s},  \eqref{KronLimas s to 0} and \eqref{KronLimitPArGen}. After multiplication, we get, as $s \to 0$, the expression
\begin{multline} \label{parabolic sum}
 \mathcal{E}^{\mathrm{par}}_{\infty}(z,s) \sum_{w\in Z(f_{2k})} h_w(s) \mathcal{E}^{\mathrm{par}}_{\infty}(w,1-s)
 \\ = \sum_{w\in Z(f_{2k})} C_w \left( 1 + s  \left[2-\log2 - \beta_M \vol(M) + \log |\eta_{\infty}^4(w) \Im (w) | +
 |\eta_{\infty}^4(z) \Im (z) | \right]\right)+O(s^2)
\end{multline}
as $s \to 0$.  Theorem \ref{thm: factorization} yields that
\begin{equation} \label{log of prod}
\log\left( \prod_{w\in Z(f_{2k})} |H_{\Gamma}(z,w)| \right) = \log |f_{2k}(z)| - \sum_{w\in Z(f_{2k})} B_{w,\infty} - \log |b_{f_{2k}}|,
\end{equation}
where $B_{w,\infty}$ is defined by \eqref{B_e_j} for the cusp $P_l=\infty$. Finally, from formula \eqref{zeros f-la}, we get that
$$
\sum_{w\in Z(f_{2k})} C_w =k.
$$
Therefore, by inserting \eqref{B_e_j},  \eqref{log of prod} and \eqref{parabolic sum} and  into \eqref{Kronecker limit ell 2}, we immediately deduce
\eqref{ell kroneck limit one cusp}. The proof is complete.
\end{proof}

\begin{remark}\rm
In the case $\Gamma=\mathrm{PSL}_2(\Z)$, the parabolic Kronecker limit function is given by
$\eta_{\infty}(z)= \eta(z)=\Delta(z)^{1/24}$.
Then, for $k=3$ and $f_{2k}=E_6$, we have $b_{E_6}=1$ and $Z(E_6)= \{i\}$, hence
Proposition \ref{prop:Ell Kron limit one cusp} yields \eqref{elliptic Eis_at_i}. Analoguously,
for $k=2$ and $f_{2k}=E_4$, we have $b_{E_4}=1$ and $Z(E_4)= \{\rho\}$, and Proposition \ref{prop:Ell Kron limit one cusp}
gives \eqref{elliptic Eis_at_rho}. Furthermore (see \cite{vP10}, p.~131),
we have $B_{E_6,\Gamma}=\exp(B_i)$ and $B_{E_4,\Gamma}=\exp(B_{\rho})$, where $B_i$ and $B_{\rho}$ are given by \eqref{elliptic_at_i}
and \eqref{elliptic_at_rho} respectively.
\end{remark}

Let us now develop further examples of a surfaces with one cusp and explicitly compute the
constant $B_{E_{2k, \Gamma}}$ in these special cases.

\subsection{Moonshine groups of square-free level}

\begin{example}\rm
Consider the surface $X_2$.  There exists one elliptic point of order two, $e_1=i/\sqrt{2}$, and one elliptic point
of order four, $e_2=1/2 + i/2$.  The surface $X_2$ has genus zero and one cusp, hence $\vol_{\hyp}(X_2)=\pi/2$.
The transformation rule for $E_6^{(2)}$ implies that the form must vanish at the points $e_1$ and $e_2$.  Furthermore, formula \eqref{zeros f-la}
when applied to $X_{2}$ becomes
\begin{align} \label{zeros f-la N=2}
\frac{2k}{8}= v_{\infty}(f) + \frac{1}{4}v_{e_2}(f)+ \frac{1}{2} v_{e_1}(f) + \sum_{z\in X_2 \setminus \{e_1,e_2\}} v_{z}(f).
\end{align}
Taking $k=3$, we conclude that $e_1$ and $e_2$ are the only vanishing points of $E_6^{(2)}$ and the order of vanishing
is one at each point.  Therefore, in the notation of Theorem \ref{thm: factorization} and Example \ref{ex: constants B_N}, we have that the form $H_6^{(2)}(z)= H_{E_6^{(2)}}(z)$ is given by
$H_6^{(2)} (z) := H_2(z, e_1)H_2(z, e_2)$. Assuming that the phase of $H_6^{(2)} (z)$ is such that it attains real values at the cusp $\infty$, we have that
\begin{equation} \label{E 6,2}
E_6^{(2)}(z) = C_{2,6} H_6^{(2)} (z),
\end{equation}
where the absolute value of the constant $C_{2,6}$ is given by
$|C_{2,6}|=e^{B_{2,e_1}+B_{2,e_2}}$ with
$$
B_{2,e_1}=-2 \left( 24\zeta'(-1) + \log(8\pi^2)
- \frac{4}{3} \log 2 +\frac{1}{12} \log\left( \left| \Delta(i\sqrt{2}) \cdot \Delta(i/\sqrt{2}) \right| \right)\right)
$$
and
$$
B_{2,e_2}=- \left( 24\zeta'(-1) + \log(8\pi^2)
- \frac{11}{6} \log 2 +\frac{1}{12} \log\left( \left| \Delta(1/2 + i/2) \cdot \Delta(1+i) \right| \right)\right).
$$

Let us now consider the case when $k=2$.  From \eqref{zeros f-la N=2}, we have that only $e_1$ and $e_2$ can be vanishing points of $E_4^{(2)}$.  However,
there are two possibilities: Either $e_2$ is an order two vanishing point, and $E_4^{(2)}(z)\neq 0$ for all $z\neq e_2$ in a fundamental domain $\mathcal{F}_2$ of $X_2$, or $e_1$ is an order one
vanishing point and $E_4^{(2)}(z)\neq 0$ for all points $z\neq e_1$ in $\mathcal{F}_2$.
If the latter possibility is true, then $E_6^{(2)}(z) /E_4^{(2)}(z)$ would be a weight $2$ holomorphic modular form which vanishes only at $e_2$, which is not possible
since there is no weight two modular form on $X_N$ for any squarefree $N$ such that the surface $X_N$ has genus zero; see \cite{JST14}. Therefore, $E_4^{(2)}$ vanishes at $e_{2}$ of order two, and
there are no other vanishing points of $E_4^{(2)}$ on $X_2$.

Hence, in the notation of Theorem \ref{thm: factorization}, we have
 $ H_4^{(2)} (z):= H_{E_4^{(2)}}(z) = H_2(z, e_2)^2$, implying that
\begin{equation} \label{E 4,2}
 E_4^{(2)}(z) =C_{2,4}H_2(z, e_2)^2,
\end{equation}
where $|C_{2,4}|=e^{2B_{2,e_2}}$. This proves that $H_2(z, e_2)^2$ is a weight four holomorphic modular function on $\overline{\Gamma_0(2)^+}$.
If we combine \eqref{E 6,2} with \eqref{E 4,2} we get
$$
H_2(z, e_1)^2= \frac{C_{2,4}}{C_{2,6}^2} \cdot\frac{( E_6^{(2)}(z))^2}{ E_4^{(2)}(z)};
$$
in other words, $H_2(z, e_1)^2$ is a weight eight holomorphic modular function on $\overline{\Gamma_0(2)^+}$.

Furthermore, application of Proposition \ref{prop:Ell Kron limit one cusp} with $f_{2k} = E_4^{(2)}$ and $Z_{f_{2k}}=\{ e_2\}$ (with multiplicity two) together with Example \ref{ex: moonshine groups} and the representation formula \eqref{E_k, p proposit fla} yield \eqref{ell Eis at e_2}.

By applying Proposition \ref{prop:Ell Kron limit one cusp} with $f_{2k} = E_6^{(2)}$ and $Z_{f_{2k}}=\{e_1, e_2\}$ together with formula \eqref{ell Eis at e_2} we get the elliptic Kronecker limit formula for $ \mathcal{E}^{\mathrm{ell}}_{e_1}(z,s)$
$$
 \mathcal{E}^{\mathrm{ell}}_{e_1}(z,s)=-s\log\left( |E_6^{(2)}(z)| |E_4^{(2)}(z)|^{-1/2} |\Delta(z) \Delta(2z)|^{-1/6}\right) + O(s^2) \text{    as    } s \to 0.
$$
\end{example}

\vskip .06in
\begin{example} \rm
Consider the surface $X_5$.  There exist three order two elliptic elements, namely $e_1=i/\sqrt{5}$, $e_2=2/5 + i/5$, and
$e_3=1/2 + i/(2\sqrt{5})$.  The surface $X_{5}$ has genus zero and one cusp, hence $\vol_{\hyp}(X_5)=\pi$.
 Using the transformation rule for $E_6^{(5)}$, one concludes that the holomorphic form
$E_6^{(5)}$ must vanish at $e_1$, $e_2$ and $e_3$.
By the dimension formula \eqref{zeros f-la}, one sees that
$e_1$, $e_2$ and $e_3$ are the only zeros of $E_6^{(5)}$.
Theorem \ref{thm: factorization} then implies that
\begin{equation}\label{E_6_5}
E_6^{(5)}(z)= C_{5,6}H_{6}^{(5)}(z),
\end{equation}
where the absolute value of the constant $C_{5,6}$ is given by
$|C_5|=e^{B_{5,e_1}+B_{5,e_2}+B_{5,e_3}}$ and
\begin{multline*}
B_{5,e_1}+B_{5,e_2}+B_{5,e_3}= -3\left(24\zeta'(-1) + \log (8\pi^2)\right) - \log 50 \\
 +\frac{1}{12}\log \left(\abs{\Delta(i/\sqrt{5})\Delta(i\sqrt{5}) \Delta(2/5 + i/5) \Delta(2+i) \Delta(1/2 + i/(2\sqrt{5})) \Delta(5/2 + i\sqrt{5}/2)}\right) .
\end{multline*}
One can view (\ref{E_6_5}) as analogue of the Jacobi triple product formula.
\end{example}


\vskip .06in
\begin{remark} \rm
Let $N=p_1 \cdot \ldots \cdot p_r$ be a squarefree number.  Then the surface $X_{N}$ has one cusp.
Numerous results are known concerning the topological structure of $X_{N}$; see, for example, \cite{Cum04} and references therein.
As a consequence, one can develop a number of results similar to the above examples when $N=2$ or $N=5$.  In particular,
Theorem \ref{thm: factorization} holds, so one can factor any holomorphic Eisenstein series $E_{2k}^{(N)}$ of weight $2k$ into
a product of elliptic Kronecker limit functions, up to a factor of modulus one.
\end{remark}

\subsection{Congruence subgroups of prime level}

Consider the surface $M_p$ for a prime $p$.  The smallest positive integer $k$ such that there exists a weight $2k$ holomorphic form is $k=1$. As a result, we have the following corollary of Theorem \ref{thm: factorization}.

\begin{corollary}
Let $f_{2k,p}$ denote weight $2k\geq 2$ holomorphic form on the surface $M_p = \overline{\Gamma_0(p)}\setminus \h$ bounded at cusps and such that the constant term in its $q-$expansion is equal to $b_{f_{2k},p}$. Then,
$$
f_{2k,p}(z)= a_{f_{2k},p} \widetilde{B}_{f_{2k},p} \prod_{w \in Z(f_{2k},p)} \widetilde{H}_p(z,w),
$$
where $a_{f_{2k},p}$ is a complex constant of modulus one and
$$
\widetilde{B}_{f_{2k},p}=\abs{b_{f_{2k},p}} \prod_{w \in Z(f_{2k},p)} \left(\exp \left[-C_w \left( \frac{2p^2 \log p}{p^2-1} + C \right) \right]
 \abs{\sqrt[p-1]{\frac{\eta(p w) ^p}{\eta(w)}} \, \Im (w)}^{-C_w}\right)
$$
with $C:=\log (8\pi ^2) +24\zeta'(-1)$.
\end{corollary}

Let us now compute the constants $\widetilde{B}_{f_{2k},p}$ for two cases.

\begin{example}\rm
If $p=2$, then the surface $M_2$ has only one elliptic point, $e=1/2 + i/2$, which has order two.
Furthermore, $\vol_{\hyp}(M_p) = \pi$, hence formula \eqref{zeros f-la} with $k=1$ implies that the holomorphic form $E_{2,2}$ defined by \eqref{E_2,p} with $p=2$
vanishes only at $e$, and the vanishing is to order one. From the $q-$expansion \eqref{q-exp E_2,p} we have that $\abs{b_{E_{2,2},2}}=2-1=1$. Since $C_e = 1$, we get
$$
E_{2,2}(z)=a_2 \cdot \frac{1}{16  \sqrt[3]{4}\,\pi^2} \exp(- 24\zeta'(-1))   \abs{\frac{\eta (1/2 + i/2)}{\eta(1+i)^2}} \widetilde{H}_2(z,e),
$$
for some complex constant $a_2$ of modulus one.
In other words, the elliptic Kronecker limit function $\widetilde{H}_2(z,e)$ is a weight two modular form
on $\overline{\Gamma_0(2)}$.
\end{example}

\begin{example}\rm
If $p=3$, then the surface $M_3$ has only one elliptic point $e=1/2 + \sqrt{3}i/6$, which has order three. The volume of the surface $M_3$ is $4 \pi/3$, hence formula \eqref{zeros f-la} with $k=1$ implies that the holomorphic form $E_{2,3}$ vanishes only at $e$, of order two. Furthermore, $\abs{b_{E_{2,2},2}}=2$ and $C_e =1/2$,
so then
$$
E_{2,3}(z)= a_3 \cdot \frac{1}{12 \sqrt[4]{27}\, \pi^2} \exp(-24 \zeta'(-1)) \abs{\sqrt{\frac{\eta\left( 1/2 + i\sqrt{3}/6\right)}{\eta\left( 3/2 + i\sqrt{3}/2\right)^3}}}\widetilde{H}_3(z,e)^2,
$$
for some complex constant $a_3$ of modulus one.
\end{example}

\section{Additional considerations}

In this section, we use the elliptic Kronecker's limit function to prove Weil's reciprocity law.  In addition, we state various
concluding remarks.

\subsection{Weil reciprocity}

To conclude this article, we will use equation \eqref{Kronecker_elliptic} to prove Weil's reciprocity law which, for
the convenience of the reader, we now state.

\vskip .10in
\begin{theorem}{\bf [Weil Reciprocity]}\label{Weil_reciprocity}
Let $f$ and $g$ be meromorphic functions on the smooth, compact Riemann surface $M$.
Let $D_{f}$ and $D_{g}$ denote the divisors of $f$ and $g$, respectively, which we write as
$$
D_{f} = \sum m_{f}(P)P \,\,\,\,\,\textrm{and} \,\,\,\,\,D_{g} = \sum m_{g}(P)P.
$$
Then
$$
\prod\limits_{w_{j}\in D_{g}}f(w_{j})^{m_{g}(w_{j})}
= \prod\limits_{z_{i}\in D_{f}} g(z_{i})^{m_{f}(z_{i})}.
$$
\end{theorem}

\vskip .10in
\begin{proof}
Consider the function
$$
I(s;f,g) = \sum\limits_{z_{i}\in D_{f}}\sum\limits_{w_{j}\in D_{g}}m_{f}(z_{i})m_{g}(w_{j})
{\cal E}^{\textrm{ell}}_{w_{j}}(z_{i},s).
$$
We shall compute the asymptotic expansion of $I(s;f,g)$ near $s=0$.
Since both $D_{f}$ and $D_{g}$ have degree zero, we immediately have
the equations
$$
\sum\limits_{z_{i}\in D_{f}}\sum\limits_{w_{j}\in D_{g}}m_{f}(z_{i})m_{g}(w_{j}) c = 0
$$
and
$$
\sum\limits_{z_{i}\in D_{f}}\sum\limits_{w_{j}\in D_{g}}m_{f}(z_{i})m_{g}(w_{j})
\log\left((\textrm{\rm Im}(z_{i}))^{c}\right)=0.
$$
Since $M$ is assumed to be smooth and compact, the terms in \eqref{Kronecker_elliptic}
involving the parabolic Eisenstein series do not appear.
Hence, we have the asymptotic expansion
\begin{align}\label{exp_I}
I(s;f,g) = -\sum\limits_{z_{i}\in D_{f}}\sum\limits_{w_{j}\in D_{g}}m_{f}(z_{i})m_{g}(w_{j})
\log\left(\vert H(z_i,w_j)\vert\right)\cdot s+ O(s^{2}) \,\,\,\,\,\textrm{\rm as $s \rightarrow 0$.}
\end{align}
Weil's reciprocity formula will be proved by evaluating
$$
\lim\limits_{s \rightarrow 0}s^{-1}I(s;f,g)
$$
in two different ways, one by first summing over the points in $D_{f}$ the sum over the points in $D_{g}$,
and the second way obtained by interchanging the order of summation.

To begin, we claim there exist constants $a_{f}$ and $a_{g}$ such that
$$
f(w) = a_{f}\prod\limits_{z_{i}\in D_{f}}H(z_{i},w)^{m_{f}(z_{i})}
\,\,\,\,\,\textrm{and}\,\,\,\,\,
g(z) = a_{g}\prod\limits_{w_{j}\in D_{g}}H(z,w_{j})^{m_{g}(w_{j})}.
$$
Indeed, both sides of each proposed equality are meromorphic functions with the
same divisors, hence, differ by a multiplicative constant.  Since both $D_{f}$ and $D_{g}$ have
degree zero, one has that
$$
\prod\limits_{z_{i}\in D_{f}}\vert a_{g}\vert^{m_{f}(z_{i})}
=\prod\limits_{w_{j}\in D_{g}}\vert a_{f}\vert^{m_{g}(w_{j})} = 1.
$$
Therefore, we can write the
lead term in \eqref{exp_I} in two ways,
yielding the identity
\begin{equation}\label{Weilabs}
\prod\limits_{w_{j}\in D_{g}}\vert f(w_{j})\vert^{m_{g}(w_{j})}
= \prod\limits_{z_{i}\in D_{f}}\vert g(z_{i})\vert^{m_{f}(z_{i})}.
\end{equation}
It remains to argue that (\ref{Weilabs}) holds without the absolute value signs, which can be
completed as follows.  First, apply the above arguments in a fundamental domain $\cal F$
of $M$
whose interior contains the support of $D_{f}$ and $D_{g}$.  On such a domain, one can choose a well-defined
branch of $H(z,w)$, hence we arrive at the equality
\begin{equation}\label{Weil}
\prod\limits_{w_{j}\in D_{g}}f(w_{j})^{m_{g}(w_{j})}
= \prod\limits_{z_{i}\in D_{f}} g(z_{i})^{m_{f}(z_{i})}
\end{equation}
viewing all points $z_{i}$ and $w_{j}$ as lying in $\cal F$.  Now, when tessellating by $\eta \in \Gamma$,
one introduces multiplicative factors of the form
\begin{equation}\label{multfactor}
\prod\limits_{w_{j}\in D_{g}}\epsilon_{\Gamma}(\eta)^{m_{g}(w_{j})}
\,\,\,\,\,\textrm{and}\,\,\,\,\,
\prod\limits_{z_{i}\in D_{f}} \epsilon_{\Gamma}(\eta)^{m_{f}(z_{i})}.
\end{equation}
Since $D_{f}$ and $D_{g}$ are degree zero, each term in (\ref{multfactor}) is equal to one.  Therefore,
one gets a well-defined extension of (\ref{Weil}) to all $z, w \in \mathbb{H} $, which completes the proof
of Theorem \ref{Weil_reciprocity}.
\end{proof}

\subsection{Unitary characters and Artin formalism}

As with parabolic Eisenstein series, one can extend the study of elliptic Eisenstein series to include
the presence of a unitary character.  More precisely, let $\pi: \Gamma \rightarrow U(n)$ denote an
$n$-dimensional unitary representation of the group $\Gamma$ with associated character $\chi_{\pi}$.  Let
us define
\begin{equation}\label{ell_eisen_pi}
{\cal E}^{\textrm{ell}}_{w}(z,s;\pi) =\sum\limits_{\eta \in  \Gamma} \chi_{\pi}(\eta)\sinh(d_{\mathrm{hyp}}(\eta z, w))^{-s}
\end{equation}
to be the elliptic Eisenstein series twisted by $\chi_{\pi}$.  Note that if $n=1$ and $\pi$ is trivial, then
the above definition is equal to $\textrm{ord}(w)$ times the series in (\ref{ell_eisen}).  (Again, we kept the
definition (\ref{ell_eisen}) in order to be consistent with the notation in \cite{JvPS14}).  In general terms, the meromorphic
continuation of (\ref{ell_eisen_pi}) can be studied using the methodology of \cite{JvPS14}, which depended on the spectral
expansion and small time asymptotics of the associated heat kernel.  As a result, we feel it is safe to say that one
subsequently can prove the continuation of (\ref{ell_eisen_pi}).

Having established the meromorphic continuation of (\ref{ell_eisen_pi}), one then can study the elliptic Kronecker limit
functions.  It would be interesting to place the study in the context of the Artin formalism relations (see \cite{JLa94} and
references therein).  The system of elliptic Eisenstein series associated to the representations $\pi$ will satisfy additive
Artin formalism relations, and, through exponentiation, the corresponding elliptic Kronecker limit functions will satisfy
multiplicative Artin formalism relations.  It would be interesting to carry out these computations in the setting of the
congruence groups $\Gamma_{0}(N)$ as subgroups of the moonshine groups $\Gamma_{0}(N)^{+}$, for instance, in order to
relate the above-mentioned computations for parabolic Kronecker limit functions.  It is possible that a similar approach
could yield further relations amongst the elliptic Kronecker limit functions.

\subsection{The factorization theorem in other cases}

\subsubsection{\bf Factorization for compact surfaces} \rm
If $M$ is compact then, in a sense, Theorem \ref{thm: factorization} becomes the following.  In the notation
of the proof Theorem \ref{thm: factorization}, the quotient
$$
F_{f_{2k}}(z):= \frac{H_{f_{2k}}(z)}{f_{2k}(z)}
$$
is a non-vanishing, bounded, holomorphic function on $M$, hence is constant, thus
$$
f_{2k}(z) = c_{f_{2k}} H_{f_{2k}}(z):= c_{f_{2k}} \prod_{w \in Z(f_{2k})} H_{\Gamma}(z,w)
$$
for some constant $c_{f}$.  The point now is to develop a strategy by which one can evaluate $c_{f_{2k}}$.  Perhaps the
most natural approach would be to study the limiting value of
$$
\widetilde{H}_{\Gamma}(z) := \lim\limits_{w \rightarrow z} \frac{H_{\Gamma}(z,w)}{z-w},
$$
which needs to be considered in the correct sense as a holomorphic form on $M$.  One can then express $c_{f_{2k}}$ in terms
of the first non-zero coefficient of $f_{2k}$ about a point $z \in Z(f_{2k})$, a product of the forms $H_{f_{2k}}(z,w)$ for
two different points in $Z(f_{2k})$ and $\widetilde{H}_{\Gamma}(z)$.  Such formulae could be quite interesting in various
cases of arithmetic interest.  We will leave the development of such identities for future investigation.

\subsubsection{\bf Factorization for surfaces with more than one cusp} \rm
It is evident that one can generalize Theorem \ref{thm: factorization} to the case when the holomorphic
form $f_{2k}$ vanishes in a cusp, or several cusps.  In such an instance, one includes factors of the parabolic Kronecker limit function in
the construction of $H_{f_{2k}}$.  The parabolic Kronecker limit function is bounded and non-vanishing in any cusp other than the
one to which it is associated, and the (fractional) order to which it vanishes follows from Theorem 1 of \cite{Ta86}.  As with Theorem
\ref{thm: factorization}, one can express any holomorphic modular form as a product of parabolic and elliptic Kronecker limit functions,
up to a multiplicative constant.  Furthermore, the multiplicative constant can be computed, up to a factor of modulus one, from the
value of the various functions at a cusp.

\vspace{5mm}
\noindent

\noindent
Jay Jorgenson \\
Department of Mathematics \\
The City College of New York \\
Convent Avenue at 138th Street \\
New York, NY 10031
U.S.A. \\
e-mail: jjorgenson@mindspring.com

\vspace{5mm}
\noindent
Anna-Maria von Pippich \\
Fachbereich Mathematik \\
Technische Universit\"at Darmstadt \\
Schlo{\ss}gartenstr. 7 \\
D-64289 Darmstadt \\
Germany \\
e-mail: pippich@mathematik.tu-darmstadt.de

\vspace{5mm}\noindent
Lejla Smajlovi\'c \\
Department of Mathematics \\
University of Sarajevo\\
Zmaja od Bosne 35, 71 000 Sarajevo\\
Bosnia and Herzegovina\\
e-mail: lejlas@pmf.unsa.ba

\begin{thebibliography}{JKvP10}

\bibitem[SCM66]{SCM66}
\emph{Seminar on complex multiplication},
eds.\ A.~Borel, S.~Chowla, C.~S.\ Herz, K.~Iwasawa, J.~P.\ Serre,
Lecture Notes in Mathematics \textbf{21} Springer-Verlag, Berlin-New York, 1966.

\bibitem[Cum04]{Cum04}
Cummins, C.: \emph{Congruence subgroups of groups commensurable with $\PSL(2,\Z)$ of genus $0$ and $1$}.
Experiment. Math. \textbf{13} (2004), 361--382.

\bibitem[De91]{De91}
Deligne, P.:
Le symbole mod\'er\'e.
Inst. Hautes \'Etudes Sci. Publ. Math. No. \textbf{73} (1991), 147-–181.

\bibitem[Fa07]{Fa07}
Falliero, T.: \emph{D\'eg\'en\'erescence de s\'eries d'Eisenstein hyperboliques}.
Math. Ann. \textbf{339} (2007), 341--375.

\bibitem[GJM08]{GJM08}
Garbin, D., Jorgenson, J. and Munn, M.: \emph{On the appearance of Eisenstein series through degeneration}.
Comment. Math. Helv. \textbf{83} (2008), 701--721.

\bibitem[GvP09]{GvP09}
Garbin, D. and von Pippich, A.-M.: \emph{On the behavior of Eisenstein series through elliptic degeneration}.
Comm. Math. Phys. \textbf{292} (2009), 511--528.

\bibitem[Go73]{Go73} Goldstein, L. J.: \emph{Dedekind sums for a Fuchsian group} I,
Nagoya Math. J. \textbf{80} (1973), 21-47.

\bibitem[GH78]{GH78}
Griffiths, P. and Harris, J.: \emph{Principles of algebraic geometry}.
John Wiley $\&$ Sons, New York, 1978.

\bibitem[He83]{He83}
Hejhal, D.: \emph{The Selberg trace formula for ${\rm PSL}(2,\mathbb{R})$.} II,
Lecture Notes in Math. \textbf{1001}, Springer-Verlag, Berlin, 1983.

\bibitem[Hel66]{Hel66}
Helling, H.:  Bestimmung der Kommensurabilit\"atsklasse der Hilbertschen Modulgruppe.
Math.\ Z.\ \textbf{92} (1966), 269--280.

\bibitem[Iwa02]{Iwa02} Iwaniec, H.: \emph{Spectral methods of automorphic forms}.
Graduate Studies in Mathematics \textbf{53}, American Mathematical Society, Providence, RI, 2002.

\bibitem[JKvP10]{JKvP10} Jorgenson, J., Kramer, J. and von Pippich, A.-M.:
\emph{On the spectral expansion of hyperbolic Eisenstein series}.
Math. Ann. \textbf{346} (2010), 931--947.

\bibitem[JLa94]{JLa94} Jorgenson, J. and Lang, S.:
\emph{Artin formalism and heat kernels}. J. Reine Angew. Math. \textbf{447} (1994) 165--200.

\bibitem[JvPS14]{JvPS14} Jorgenson, J., von Pippich, A.-M. and Smajlovi\'c, L.:
\emph{On the wave representation of hyperbolic, elliptic, and parabolic Eisenstein series}, submitted for publication.

\bibitem[JST14]{JST12}
Jorgenson, J, Smajlovi\'c, L. and Then, H.:
\emph{On the distribution of eigenvalues of Maass forms on certain moonshine
groups}.  Math. Comp. \textbf{83} (2014), 3039--3070.

\bibitem[JST14a]{JST13}
Jorgenson, J, Smajlovi\'c, L. and Then, H.: \emph{Kronecker's limit formula, holomorphic modular
functions and $q$-expansions on certain moonshine groups}, submitted for publication.
 preprint.

\bibitem[JST15]{JST14}
Jorgenson, J, Smajlovi\'c, L. and Then, H.: \emph{
Aspects of holomorphic function theory on genus zero moonshine groups}, in preparation.

\bibitem[Kh08]{Kh08}
Khovanskii, A.: Logarithmic functional and reciprocity laws. In: \it  Toric topology, \rm
M. Harada, Y. Karshon, M. Masuda, and T. Panov, eds. Contemp. Math., \textbf{460}, Amer. Math. Soc.,
Providence, RI, (2008), 221--229.

\bibitem[Ku73]{Ku73}
Kubota, T.: \emph{Elementary theory of Eisenstein series.} Kodansha
Ltd., Tokyo, 1973.

\bibitem[KM79]{KM79}
Kudla, S. S. and Millson, J. J.: \emph{Harmonic differentials and closed geodesics on a Riemann surface}.
Invent. Math. \textbf{54} (1979), 193--211.

\bibitem[La76]{Lang76}
Lang, S. \emph{Introduction to modular forms}.
Grundlehren der mathematischen Wissenschaften \textbf{222}, Springer-Verlag, Berlin, 1976.

\bibitem[La88]{Lang88}
Lang, S.: \emph{Introduction to Arakelov theory}.  Springer-Verlag, Berlin, 1988.

\bibitem[vP10]{vP10} von Pippich, A.-M.: \emph{The arithmetic of elliptic Eisenstein
series}. PhD thesis, Humboldt-Universit\"{a}t zu Berlin, 2010.

\bibitem[vP15]{vP15}
von Pippich, A.-M.: \emph{A Kronecker limit type formula for elliptic Eisenstein series},
in preparation.


\bibitem[Se73]{Se73}
Serre, J.-P.: \emph{A Course in Arithmetic},
Graduate Texts in Mathematics, \textbf{7},
Springer-Verlag, New York, 1973.

\bibitem[Sh71]{Sh71}
Shimura, G.: \emph{Introduction to the Arithmetic Theory of Automorphic Forms},
Publications of the Mathematical Society of Japan,
Princeton University Press, Princeton, 1971.

\bibitem[Si80]{Siegel80}
Siegel, C. L.: \emph{Advanced analytic number theory.}
Tata Institute of Fundamental Research Studies in Mathematics,
\textbf{9}, Tata Institute of Fundamental Research, Bombay, 1980.

\bibitem[Ta86]{Ta86}
Takada, I.: \emph{Dedekind sums of $\Gamma(N)$}.
Japan J. Math. \textbf{12} (1986), 401--411.


\bibitem[Va96]{Vassileva96} Vassileva,  I. N.: \emph{Dedekind eta function, Kronecker limit formula and
Dedekind sum for the Hecke group}, PhD thesis, University of Massachusetts Amherst, 1996.

\end{thebibliography}
\end{document}